\documentclass[final,1p,times,authoryear]{elsarticle}
\usepackage{amsmath}
\usepackage{amsthm}
\usepackage{amssymb}
\usepackage{amsfonts}
\usepackage{dsfont}
\usepackage{algorithmic}
\usepackage{float}
\usepackage[plainpages=false,pdfpagelabels=true,colorlinks=true,citecolor=blue,hypertexnames=false]{hyperref}
\usepackage[ruled,vlined]{algorithm2e}
\usepackage{mathtools}
\usepackage{url}
\usepackage{multirow}
\usepackage{color}
\usepackage{diagbox}
\usepackage{booktabs}

\newtheorem{lemma}{Lemma}[section]
\newtheorem{theorem}{Theorem}[section]
\newtheorem{definition}{Definition}[section]

\newtheorem{example}{Example}[section]

\begin{document}

\begin{frontmatter}

\title{Equivalent of Multivariate Polynomial Matrix}

\author[a]{Jinwang Liu}\ead{jwliu64@aliyun.com}
\author[b]{Tao Wu$^*$}\ead{839417221@qq.com}

\author[a]{Jiancheng Guan}\ead{jiancheng-guan@aliyun.com}
\author[a]{Ying Kang}\ead{1343317573@qq.com}

\address[a]{School of Mathematics and Computational Science, Hunan University Science and Technology, XiangTan, 411201, China}
\address[b]{School of Computer Science and Engineering, Hunan University Science and Technology, XiangTan, 411201, China}

\tnotetext[labeltitle]{$^*$Corresponding author.}

\begin{abstract}
The equivalence of multidimensional systems is closely related to the reduction of multivariate polynomial matrices, with the Smith normal form of matrices playing a key role. So far, the problem of reducing multivariate polynomial matrices into Smith normal form has not been fully solved. This article investigates the Smith normal form of quasi weakly linear multivariate polynomial matrices. Breaking through previous limitations, providing sufficient and necessary conditions and algorithms.

\end{abstract}

\end{frontmatter}

\section{Introduction}

Multidimensional ($n$D) linear systems are very common in the solution of engineering problems, and multivariate polynomial matrices are the most basic tools for studying multidimensional systems.  The Smith normal form of multivariate polynomial matrices plays an important role in solving problems such as circuit analysis, coding, network communication, and automation control.

The nD polynomial matrix theory is a well established tool for the study of these systems, since some problems in the analysis and synthesis of control systems can be solved using polynomial matrix techniques  \cite{Bose1982,Quillen1976,Suslin1976,Kailath1980,Bose2003,Becker1994,Pommaret2001,Quadrat2010,Antsaklis2006,Boudellioua2012}. It is well known that these linear systems, both finite and infinite, can be studied from the algebraic point of view, see for instance \cite{Callier1982,Vardulakis1991,Boudellioua2010,Chu2012,Vologiannidis2016,Srinivas2004,William1993} and the references therein. The starting point of the analysis conducted is the transfer function matrix arising from the input-output description of a system, the elements of such matrices are rational polynomial in a number of indeterminate. Two commonly used forms of describing the transfer function matrix are the matrix fraction description and the Rosenbrock system matrix. The state space method plays an important role in the modern theory of automatic control where matrices over commutative polynomial rings are used to represent linear systems, an idea can at least trace back to Rosenbrock (1970) \cite{Rosenbrock1970}. To see this, let us revisit. The dynamic behavior of a string with an interior mass can be depicted by the equations

\begin{equation}\label{SM}
\begin{cases}
\phi_1(t)~+\psi_1(t)-~\phi_2(t)-~\psi_2(t)=0,&\\
\dot{\phi_1}(t)+\dot{\psi_1}(t)~+\eta_1\phi_1(t)-\eta_1\psi_1(t)-\\~~~~~~~~~~~~~~~~~~~~\eta_2\phi_2(t)+\eta_2\psi_2(t)=0,&\\
\phi_1(t-2h_1)+\psi_1(t)-u(t-h_1)=0,&\\
\phi_2(t)+\psi_2(t-2h_2)-v(t-h_2)=0.
\end{cases}
\end{equation}
here, $h_1,\,h_2,\,\eta_1,\,\eta_2$ are positive constants, as
usual, the dot indicates the time derivative (see \cite{Rosenbrock1970}
for more detailed physical interpretations of the notions).
Introducing the differential operator $\partial=\frac{\mathrm
d}{\mathrm dt}$ and the two time-delay operators
$\sigma_1,\,\sigma_2$ given by
$$
\partial f(t):=\dot f(t),\quad\sigma_1f(t):=f(t-h_1),\quad
\sigma_2f(t):=f(t-h_2)
$$
leads to the commutative (the commutativeness stems from the fact
that the differential operator and the delay operator swap)
multivariate polynomial ring $R=\mathbb R[\partial, \sigma_1,
\sigma_2]$ and the multivariate polynomial matrix
$$T=\begin{bmatrix}
1& ~1 & -1 & -1 & 0 ~& 0\\
\partial+\eta_1~ & ~\partial-\eta_1~ &~\eta_2 ~&~ \eta_2~ & 0 ~& 0\\
\sigma_1^2 ~& 1 ~& 0~& 0 ~& -\sigma_1~ & 0\\
0 ~& 0 ~& 1~ & \sigma_2^2 ~& 0 ~& -\sigma_2
\end{bmatrix}$$
Consequently, the original delay-differential equations system are
reformulated into an abstract linear system of the form
$$
T\cdot x(t)=0
$$
where $x(t)=[\phi_1(t)\,\,\, \psi_1(t)\,\,\, \phi_2(t)\,\,\,
\psi_2(t)\,\,\, u(t)~\, v(t)]^T$ indicates the system state at
time $t$.

 It is not surprising that equivalent reductions of $T$
simplify the investigation of the original system. More often than
not, the reduced new system contains fewer equations or unknowns.
Among which a typical reduction is Serre's  reduction that a matrix is equivalent to its Smith normal form. For
instance, if one can show that a square multivariate polynomial
matrix $T\in R^{\,l\times l}$ is equivalent to the Smith normal
form
$$
S=\begin{bmatrix}
I_{l-1} & 0\\
0 & {\rm det} T
\end{bmatrix},
$$
then the reduced equivalent system contains only one equation in one
unknown. The reduced matrix $S$ has good properties and corresponds to the system whose interesting information
 can be   extracted easily,
 which does not  appear clearly in the original forms.

Over the past few decades, it is well known that these smith normal form of polynomial matrices  be investigated from many scholars\cite{Mounier1998,Boudellioua2013,Frost1979,Lee1983,Bachelier2017,Antoniou2019,Frost1986,Lin1999,Lin2001,Lin2005,Lin2006,Li2017,Lu2017,Liu2014,Wang2007,Wang2008,Lu2017,Zheng2023} for instance: In 1970, Rosenbrock confirmed that all univariate polynomials are equivalent to its smith normal form \cite{Rosenbrock1970}. And in 1981, Frost and Storey gave an example where a binary polynomial matrix is not equivalent to its Smith normal form\cite{Frost1981}. It is extremely difficult to completely solve the equivalence problem of matrices, which means that for general two multivariate polynomial matrices, we do not have a general method to determine whether they are equivalent, and there is no general algorithm. The limitations of this theory and algorithm prompt us to study the construction of multivariate polynomial matrices, for example, weakly linear multivariate polynomial matrices, quasi weakly linear multivariate polynomial matrices, etc.

The motivation for studying structured multivariate polynomial matrices lies in the fact that the multivariate polynomial matrices that appear in practical problems often have specific intrinsic structures. For example, in some control systems, multivariate polynomials that characterize stability have the maximum common factor of their maximum order sub expressions that is linear with respect to a certain variable. This means that once efficient equivalent algorithms are proposed, we have the hope of solving one or even a class of problems. 

Lin studied the smith equivalence conditions of a particular class of multivariate polynomial matrices which $det F(\boldsymbol x)=x_1-f(x_2,\cdots,x_n)$\cite{Lin2006}, then this special matrix is improved by Li in 2017\cite{Li2017}.\\

Let $F(\boldsymbol x)\in K^{l\times m}[\boldsymbol x],(l\leq m)$ be an $n$D polynomial matrix,  $F(\boldsymbol x)$ is said to be a weakly linear polynomial matrix if the greatest common divisor (g.c.d.) of the $l\times l$ minors of $F(\boldsymbol x)$ is $d^q$, where $d=x_1-f(x_2,\cdots,x_n)$, $q$ is positive integer. This kind of weakly linear matrices is very important  in engineering application. Li studied the Smith normal forms and factorization for some weakly linear matrices, they derived some necessary and sufficient conditions on the equivalence of such matrices are all the the $(l-1)\times (l-1)$ minors generate the unit ideal [19]. However, this sufficient and necessary condition only applies to matrices with special Smith normal form where the first l-1 elements on the main diagonal of the Smith normal form are 1 and has limitations. In 2024, Liu further extended this type of matrix and provided sufficient and necessary conditions, which no longer have limitations \cite{Liu2024}.

$F(\boldsymbol x)$ is said to be a quasi weakly linear polynomial matrix if the greatest common divisor (g.c.d.) of the $l\times l$ minors of $F(\boldsymbol x)$ is $d$, where $d=(x_1-f_1(x_2,\cdots,x_n))^{q_1}(x_2-f_2(x_3,\cdots,x_n))^{q_2}$, where $q_1$ and $q_2$ are positive integers. Li provides a necessary and sufficient condition for a quasi weak linear polynomial matrix to be equivalent to its Smith normal form, where the first l-1 elements on the main diagonal of the Smith normal form are 1 [21].

So far, we find that although the equivalence of multivariate polynomial matrices has achieved quite satisfactory results, there is no hope that the equivalence problem can be solved completely. Thus it is natural for us to consider the following problems on weakly linear polynomial matrices.\\

{\bf Problem 1:}\quad Let $F(\boldsymbol x)\in  K^{l\times m}[\boldsymbol x](l\leq m)$ be of full row rank, $d_l(F)=(x_1-f_1(x_2,\cdots,x_n))^p(x_2-f_2(x_3,\cdots,x_n))^q$,
where  $d_l(F)$ is the g.c.d of the $l\times l$ minors of $F(\boldsymbol x)$,  $p,q$ are positive integers.
When is the $F(\boldsymbol x)$  equivalent to its Smith normal form?

{\bf Problem 2:}\quad Let $F(\boldsymbol x)\in  K^{l\times m}[\boldsymbol x](l\leq m)$ be of normal rank $r$,   $d_r(F)=(x_1-f_1(x_2,\cdots,x_n))^p(x_2-f_2(x_3,\cdots,x_n))^q$,
where  $d_r(F)$ is the g.c.d of the $r\times r$ minors of $F(\boldsymbol x)$,  $p,q$ are positive integers. When is the $F(\boldsymbol x)$ equivalent to its Smith normal form? Assume $F(\boldsymbol x)$ is equivalent to its Smith normal form, is there an executable algorithm?

\section{Preliminaries }
    In the following,  $R=\mathbf{K}[x_1,x_2,\cdots,x_n]$ denotes the set of polynomials in $n$ variables $x_1,x_2,\cdots,x_n$ and the coefficients in the field $\mathbf{K}$,  $\mathbf{K}[x_1,x_2,\cdots,x_n]$ is omitted as $\mathbf{K}[\boldsymbol x]$ whenever its omission does not cause confusion. $\overline{\mathbf{K}}$ denotes a algebraic closed field of $\mathbf{K}$.
$\overline{\mathbf{R}}_1\triangleq \mathbf{K}[x_2,\cdots,x_n]$,$\overline{\mathbf{R}}_2\triangleq \mathbf{K}[x_1,x_3,\cdots,x_n]$. $\mathbf{K}^{l\times m}$  denotes the set of $l\times m$ matrices with elements entries in $\mathbf{K}$ and $0_{l\times m}$ denotes the $l\times m$ zero matrix and $E_k$ denotes the $k\times k$ identity matrix.
\begin{definition}
Let $F(\boldsymbol x)\in\mathbf{K}^{l\times m}[\boldsymbol x]$ be of normal rank $r$. Denote $i\times i$ minors of $F(\boldsymbol x)$  by $\alpha_{i1},\cdots,\alpha_{ik_i}$ and denote the greatest common divisor $(g.c.d.)$ of $\alpha_{i1},\cdots,\alpha_{ik_i}$ by $d_i(F)$ where $i=1, \cdots, r$.
$$\beta_{i1}\triangleq \alpha_{i1}/d_i(F),\cdots,\beta_{ik_i}\triangleq \alpha_{ik_i}/d_i(F),$$
are called the reduced i-th order minors of $F(\boldsymbol x)$.
\end{definition}

$J_i(F)$ denotes the ideal generated by $\beta_{i1},\cdots, \beta_{ik_i}$, and  $d(F)\triangleq d_r(F)$.
\begin{definition} Let $F(\boldsymbol x)\in \mathbf{K[x]}^{l\times m},(l\leq m))$ be of full row  rank.  $F(\boldsymbol x)$ is said to be zero left prime (ZLP) if  the $l\times l$ minors of $F(\boldsymbol x)$ generate the unit ideal $\mathbf{R}$.
\end{definition}

\begin{definition} Let $F(\boldsymbol x)\in \mathbf{K[\boldsymbol x]}^{l\times m}$, $l\leq m$, the Smith normal form of $F(\boldsymbol x)$ is $$S=(diag\{\Phi_i\}\ \ 0_{l\times(m-l)}),$$
where
\begin{equation*}
\Phi_i=\left\{
\begin{aligned}
&d_i/d_{i-1},     &&1\leq i\leq r\\
&0,     &&r<i\leq m
\end{aligned}
\right.,
\end{equation*}
$r$ is the normal rank of $F(\boldsymbol x)$, $d_0\equiv 1$, $d_i$ is the g.c.d. of the $i\times i$ minors of $F(\boldsymbol x)$ and $\Phi_i$ has the following property 
 $$\Phi_1\mid\Phi_2,\ \Phi_2\mid\Phi_3,\cdots,\Phi_{r-1}\mid\Phi_r.$$
\end{definition}
\begin{definition} Let $F_1(\boldsymbol x)$, $F_2(\boldsymbol x)$$\in\mathbf{K[\boldsymbol x]}^{l\times m}$, if there exist two unimodular matrices $M(\boldsymbol x)\in K[\boldsymbol x]^{l\times l}$ and $N(\boldsymbol x)\in K[\boldsymbol x]^{m\times m}$ such that $F_2(\boldsymbol x)=M(\boldsymbol x)\cdot F_1(\boldsymbol x)\cdot N(\boldsymbol x)$, then we said $F_1(\boldsymbol x)$ and $F_2(\boldsymbol x)$ are equivalent.
\end{definition}

As long as the omission of parameter does not cause confusion, omit it. For example, we use F to represent $F(\boldsymbol x)$.
We first introduce several important lemmas.

\section{lemma}

\begin{lemma}\cite{Liu2024} Suppose $F=F_1\cdot F_2$ with $F, F_2\in \mathbf{K}^{l\times m}[\boldsymbol x], F_1\in \mathbf{K}^{l\times l}[\boldsymbol x]$. If all the $k\times k$ minors of $F$ have no common zeros for some $k (k\leq l)$, then all the $k\times k$ minors of $F_i,(i=1,2)$ have no common zeros.
\end{lemma}
\begin{lemma}\cite{Liu2024} Suppose $F\in \mathbf{K}^{l\times m}[\boldsymbol x]$ has normal rank r. If all the $r\times r$ reduced minors of $F$ generate $K[\boldsymbol x]$, then there exist a $ZLP$ polynomial matrix $H\in\overline{\mathbf{K}}^{(l-r)\times l}$ which satisfies $H\cdot F=0_{(l-r) \times m}$.
\end{lemma}

\begin{lemma}\cite{Liu2024} Let $F,A,B \in \mathbf{R}^{l\times l}$, $F=A\cdot B$. For some $k(1\leq k\leq l)$, if $d_k(A)=d_k(F), J_k(F)=\mathbf{R}$, then $J_k(A)=\mathbf{R}, d_k(B)=1,J_k(B)=\mathbf{R}$.
\end{lemma}

Remark: Let $\alpha_1,\cdots, \alpha_q$ be the $r\times r$  minors of $F(x_1,x_2,\ldots,x_n)$, and  $\beta_1,\cdots, \beta_q$ are the $r\times r$  minors of $F(0,x_2,\ldots,x_n)$;
 We assumed the first $r\times r$  minor of $F(x_1,x_2,\ldots,x_n)$ is:
 $$\begin{aligned}
&\alpha_1(x_1,x_2,\cdots ,x_n)= \\
&\begin{vmatrix}
                             a_{11}(x_1,x_2,\cdots ,x_n)&  \cdots &  a_{1r}(x_1,x_2,\cdots ,x_n)\\
                               a_{21}(x_1,x_2,\cdots ,x_n)&  \cdots &  a_{2r}(x_1,x_2,\cdots ,x_n)\\
                                \vdots&  \cdots &  \vdots\\
                              a_{r1}(x_1,x_2,\cdots ,x_n)&  \cdots &  a_{rr}(x_1,x_2,\cdots ,x_n)\\
                          \end{vmatrix}.
                         \end{aligned}$$\\
 and the first $r\times r$  minor of $F(0,x_2,\ldots,x_n)$ is:
 $$\begin{aligned}
&\beta_1(0,x_2,\cdots ,x_n)= \\
&\begin{vmatrix}
                             a_{11}(0,x_2,\cdots ,x_n)&  \cdots &  a_{1r}(0,x_2,\cdots ,x_n)\\
                               a_{21}(0,x_2,\cdots ,x_n)&  \cdots &  a_{2r}(0,x_2,\cdots ,x_n)\\
                                \vdots&  \cdots &  \vdots\\
                              a_{r1}(0,x_2,\cdots ,x_n)& \cdots &  a_{rr}(0,x_2,\cdots ,x_n)\\
                          \end{vmatrix}.
                         \end{aligned}$$
thus $\alpha_1(0,x_2,\cdots ,x_n)=\beta_1(0,x_2,\cdots ,x_n)$, generally, $\alpha_i(0,x_2,\cdots ,x_n)=\beta_i(0,x_2,x_3,\cdots ,x_n), i=1,\cdots,q$;\\

\begin{lemma} Suppose $b_1(x_1,\cdots ,x_n),\cdots, b_q(x_1,\cdots ,x_n)$;
$c_1(0,x_2,x_3,\cdots ,x_n),\cdots, c_q(0,x_2,x_3,\cdots ,x_n)$ be the $r\times r$ reduce minors of $F(x_1,x_2,\ldots,x_n)$ and $F(0,x_2,\ldots,x_n)$, respectively.
Then  $c_i(0,x_2,x_3,\cdots ,x_n)$ is a divisor of $b_i(0,x_2,\cdots ,x_n), i=1,\cdots,q$.
\end{lemma}
\begin{proof}Let $b_1^{'}(x_1,x_2,\cdots ,x_n),\cdots, b_q^{'}(x_1,x_2,\cdots ,x_n)$ are the $r\times r$  minors of $F(x_1,x_2,\ldots,x_n)$, and $d(x_1,x_2,\cdots ,x_n)$ is the g.c.d of  $b_1^{'},\cdots, b_q^{'}$. $c_1^{'}(0,x_2,\cdots ,x_n),\cdots, c_q^{'}(0,x_2,\cdots ,x_n)$ are the $r\times r$ minors of $F(0,x_2,\ldots,x_n)$,  and $d_1(0,x_2,x_3,\cdots ,x_n)$ is the g.c.d of  $c_1^{'},\cdots, c_q^{'}$. By remark, $$b_i^{'}(0,x_2,\cdots ,x_n)=c_i^{'}(0,x_2,x_3,\cdots ,x_n).$$
Since $d(x_1,\cdots, x_n) | b_i^{'}(x_1,\cdots ,x_n), i=1,\cdots,q,$ we have $d(0,x_2,\cdots ,x_n) | b_i^{'}(0,x_2,\cdots ,x_n)$. Hence $$d(0,x_2,\cdots ,x_n) | c_i^{'}(0,x_2,\cdots ,x_n).$$  
Since  $d_1(0,x_2,\cdots ,x_n)$ is the g.c.d of  $c_i^{'}(0,x_2,\cdots ,x_n) $, we have $d(0,x_2,\cdots ,x_n) | d_1 (0,x_2,x_3,\cdots ,x_n).$ So there exist $h(0,x_2,\cdots ,x_n)$ such that $$d_1(0,x_2,x_3,\cdots ,x_n)=d(0,x_2,\cdots ,x_n)\cdot h(0,x_2,\cdots ,x_n).$$
Furthermore,
 
 $d(0,x_2,\cdots ,x_n)\cdot b_i(0,x_2,\cdots ,x_n)$
 
 $=d_1(0,x_2,\cdots ,x_n)\cdot c_i(0,x_2,\cdots ,x_n)$

$=d(0,x_2,\cdots ,x_n)\cdot h(0,x_2,\cdots ,x_n)\cdot c_i(0,x_2,x_3,\cdots ,x_n),$
\\where $h(0,x_2,\cdots ,x_n)\in K[x_2,\cdots,x_n].$ \\Since $d(0,x_2,\cdots ,x_n)\neq 0$, so $c_i(0,x_2,x_3,\cdots ,x_n)$ is a divisor of $b_i(0,x_2,\cdots ,x_n)$, where $i=1,\cdots,q$.

\end{proof}
For problem 1, in order to simplify the whole proof, we first concentrate on the equivalence problem of matrix which with determinant is the product of univariates, i.e. $d(F)=x_{1}^p\cdot x_{2}^q$.

\begin{lemma} Let $F\in \mathbf{K[x_1,\cdots,x_n]}^{l\times m}$ be of normal full row rank and $d(F)=x_{1}^p\cdot x_{2}^q$, and  p,q is a positive integer. If $x_{i}\nmid d_{l-r}(F)$, $J_{l-r}(F)=\mathbf{R}$, and $x_{i}\mid d_{l-r+1}(F)$, then there exists  unimodular matrix $ W\in\overline{\mathbf{R}}_i^{l\times l} ,i=1,2$ such that\\
$$
W\cdot F= 
\left(
   \begin{array}{cccccc}
    1& & & & &\\
    & \ddots& & & &\\
    & & 1& & &\\
    & & & x_i& &\\
    & & & & \ddots&\\
    & & & & &x_i\\
 \end{array}
 \right)
 \cdot F_1;\\
$$

or If $x_{i}\nmid d_{l-r}(F),i=1,2$, $J_{l-r}(F)=\mathbf{R}$, and $x_{1}x_{2}\mid d_{l-r+1}(F)$,
then there exists  unimodular matrix $V\in\overline{\mathbf{R}}_1^{l\times l},U\in\overline{\mathbf{R}}_2^{l\times l}$ such that $U\cdot F=$\\
$$
\small\left(
   \begin{array}{cccccc}
    1& & & & &\\
    & \ddots& & & &\\
    & & 1& & &\\
    & & & x_2& &\\
    & & & & \ddots&\\
    & & & & &x_2\\
 \end{array}
 \right)
\
 \cdot V \cdot\\
\
\left(
   \begin{array}{cccccc}
   1& & & & &\\
    & \ddots& & & &\\
    & & 1& & &\\
    & & & x_1& &\\
    & & & & \ddots&\\
    & & & & &x_1\\
\end{array}
 \right)
\
 \cdot F_1,\\
$$
where $F_{1}\in \mathbf{K[x]}^{l\times m}$.
\end{lemma}
\begin{proof} We assumed $x_{2}\nmid d_{l-r}(F)$,$x_2 \mid d_{l-r+1}(F)$, it follows that rank
$F(x_1,0,x_3,\ldots,x_n) \leqslant l-r$. because $x_2\nmid d_{l-r}(F)$, we obtain that $d_{l-r}(F(x_1,0,x_3,\ldots,x_n)) \neq 0$, so $ rank(F(x_1,0,x_3,\ldots,x_n))=l-r$.

let $e_1(x_1,x_2\cdots ,x_n),\cdots,e_t(x_1,x_2\cdots ,x_n)$ and $y_1(x_1,0,x_3,\cdots ,x_n),$ $\cdots, y_t(x_1,0,x_3,\cdots ,x_n)$ 
be the $(l-r)\times(l-r)$ reduce minors of $F(x_1,\ldots,x_n)$ and $F(x_1,0,x_3,\ldots,x_n)$, respectively. According the lemma 3.4 we have that $y_i(x_1,0,x_3,\cdots ,x_n)$ is a divisor of $e_i(x_1,0,x_3,\cdots ,x_n)$, where $i=1,\cdots,t.$ Since $J_{l-r}(F)=\mathbf{R}$, $e_1,\cdots,e_t$ have no common zeros, then $e_1(x_1,0,x_3,\ldots,x_n),$\\$\cdots,e_t(x_1,0,x_3,\ldots,x_n)$  also have no common zeros, so $y_1(x_1,0,x_3,\cdots ,x_n),$ $\cdots, y_t(x_1,0,x_3,\cdots ,x_n)$ also have no common zeros, i.e. the $(l-r)\times(l-r)$ reduce minors of $F(x_1,0,x_3,\ldots,x_n)$ have no common zeros.

According to lemma 3.2, there exists a $ZLP$ polynomial matrix $H \in\overline{\mathbf{R}}_2^{r\times l}$ which satisfies $H \cdot F(x_1,0,x_3,\cdots,x_n)=0_{r \times m}$. By Quillen-Suslin Theorem, we can complete H into an $l\times l$ unimodular matrix $W(x_1,0,x_3,\cdots,x_n)$, we have that r rows of the matrix $W\cdot F(x_1,0,x_3,\cdots,x_n)$ are the zero polynomials. So, there is the common divisor $x_2$ in the r rows of $W \cdot F(x_1,0,x_3,\cdots,x_n)$. So 
$$
W \cdot F= 
\left(
   \begin{array}{cccccc}
    1& & & & &\\
    & \ddots& & & &\\
    & & 1& & &\\
    & & & x_2& &\\
    & & & & \ddots&\\
    & & & & &x_2\\
 \end{array}
 \right)
 \cdot F_1\\
$$
(the proof of $x_1$ is same as $x_2$).

If $x_{i}\nmid d_{l-r}(F),i=1,2$, $J_{l-r}(F)=\mathbf{R}$, and $x_{1}x_{2}\mid d_{l-r+1}(F)$, let $W\cdot F=D$, where\\

$$
D=
\left(
   \begin{array}{cccccc}
    1& & & & &\\
    & \ddots& & & &\\
    & & 1& & &\\
    & & & x_2& &\\
    & & & & \ddots&\\
    & & & & &x_2\\
   \end{array}
 \right)
\
 \cdot G\\
\
,$$
because W is unimodular matrix, so that $F\sim D$, according to lemma $3.3$, we have $d_k(F)=d_k(D), J_k(F)=R\Leftrightarrow J_k(D)=R$. Thus $x_1x_2\mid d_{l-r+1}(D), J_{l-r+1}(D)=R$. let $a_1(x_1,x_2,\cdots ,x_n),\cdots, a_s(x_1,x_2,\cdots ,x_n)$  be the $(l-r+1)\times(l-r+1)$ minors of G, thus $a_1^{'}(x_1,x_2,\cdots ,x_n),\cdots, a_s^{'}(x_1,x_2,\cdots ,x_n)$ be the $(l-r+1)\times(l-r+1)$ minors of D, and $ x_2|a_i^{'}(x_1,x_2,\cdots ,x_n)$. Since $x_1x_2\mid a_i^{'}$, $x_1\mid a_i$, i.e. $x_1\mid d_{l-r+1}(G)$. So rank $G(0,x_2,x_3,\ldots,x_n)\leq l-r$, and $x_1\nmid d_{l-r}(G)$, thus $d_{l-r}(G(0,x_2,x_3,\ldots,x_n)) \neq 0$, so rank $G(0,x_2,x_3,\ldots,x_n)=l-r$. We then show that $J_{l-r}(G)=\mathbf{R}$, If otherwise, all the $(l-r)\times (l-r)$ reduce minors of $G$ have common zeros, this implies that all the $(l-r)\times (l-r)$ reduce minors of $D$ have common zeros, then $J_{l-r}(D)\neq \mathbf{R}$, which leads to a contradiction.

let $b_1(x_1,x_2,\cdots, x_n),\cdots, b_q(x_1,x_2,\cdots ,x_n)$ and $c_1(0,x_2,x_3,\cdots ,x_n),$ $\cdots, c_q(0,x_2,x_3,\cdots ,x_n)$ be the $(l-r)\times(l-r)$ reduce minors of $G[x_1,x_2,\ldots,x_n]$ and $G[0,x_2,\ldots,x_n]$, respectively. It is straightforward that $c_i(0,x_2,x_3,\cdots ,x_n)$ is a divisor of $b_i(0,x_2,\cdots ,x_n),(i=1,\cdots,q)$. thus the $(l-r)\times(l-r)$ reduce minors of $G[0,x_2,\ldots,x_n]$ have no common zeros.

So ,there exists an  unimodular matrix $V_1\in\overline{\mathbf{R}}_1^{l\times l}$ such that\\
$$
V_{1}\cdot G =
\left(
   \begin{array}{cccccc}
    1& & & & &\\
    & \ddots& & & &\\
    & & 1& & &\\
    & & & x_1& &\\
    & & & & \ddots&\\
    & & & & &x_1\\
   \end{array}
 \right)
\
 \cdot F_1 \\
\
,$$
Setting $U=W,V=V_1^{-1}$, we obtain that\\
$$\small U F=
\left(
   \begin{array}{cccccc}
    1& & & & &\\
    & \ddots& & & &\\
    & & 1& & &\\
    & & & x_2& &\\
    & & & & \ddots&\\
    & & & & &x_2\\
   \end{array}
 \right)
\cdot
  V 
\cdot
\left(
   \begin{array}{cccccc}
   1& & & & &\\
    & \ddots& & & &\\
    & & 1& & &\\
    & & & x_1& &\\
    & & & & \ddots&\\
    & & & & &x_1\\
   \end{array}
 \right)
\cdot
  F_1
\
.$$
Note that U and V are unimodular matrices.
\end{proof}

\begin{lemma}
Let $F\in \mathbf{K[x]}^{l\times l}$ be of full row rank, $d(F)=x_1^{p}x_2^{q}$,  where  $p,q$ is a positive integer.  If there exist two subsets $\{i_1,\ i_2,\ \cdots\ ,\ i_m\}$ and $\{j_1,\ j_2,\ \cdots\ ,\ j_m\}$ of $\{1,\ 2,\ \cdots\ ,\ l\}$ such that $x_1\nmid$
det$F\left(
      \begin{array}{ccccc}
        i_1\ i_2\ \cdots \ i_m\\
    j_1\  j_2\ \cdots\  j_m\\
      \end{array}
    \right)$,\\
$x_1\mid$
det$F\left(
      \begin{array}{cccccc}
        i_1\ i_2 \ \cdots\ i_m\ k_{m+1}\\
p_1\ p_2\ \cdots \ p_m\ p_{m+1}\\
      \end{array}
    \right)$ for any $k_{m+1}(k_{m+1}\neq i_1,i_2,\cdots,i_m),$ and any permutation $p_1p_2\cdots p_mp_{m+1}$ of $1,2,\cdots,l$. Then $x_1\mid d_{m+1}(F)$.
\end{lemma}
\begin{proof}
We can write every entry $a_{ij}$ of $F(\boldsymbol x)$ as $a_{ij}=b_{ij}x_1+h_{ij} $,where $h_{ij}\in K[x_2,\cdots,x_n],i , j=1,2\cdots,l$. Setting $H\triangleq (h_{ij})^{l\times l},$ according to Laplace Theorem, we have
 $$detF\left(
        \begin{array}{cccc}
          q_1&q_2 &\cdots &q_k\\
          p_1&p_2&\cdots &p_k\\
        \end{array}
      \right)$$
    $$=f_kx_1^{k}+\cdots +f_1x_1^{1}+det H\left(
        \begin{array}{cccc}
          q_1&q_2 &\cdots &q_k\\
          p_1&p_2&\cdots &p_k\\
        \end{array}
      \right).      $$
Since $$x_1 \nmid detF\left(
        \begin{array}{cccc}
          i_1&i_2 &\cdots &i_m\\
          j_1&j_2&\cdots &j_m\\
        \end{array}
      \right),$$
 then $$detH\left(
        \begin{array}{cccc}
          i_1&i_2 &\cdots &i_m\\
          j_1&j_2&\cdots &j_m\\
        \end{array}
      \right)\neq 0,$$
consider row vectors $ \gamma_1,\gamma_2,\cdots,\gamma_l$ of H as in l-dimension vector space $K^{1\times l}(x_2\cdots,x_n)$,we see that $ \gamma_{i_1},\gamma_{i_2},\cdots,\gamma_{i_m}$ are linearly independent. Note that
 $$ x_1 \mid detF\left(
        \begin{array}{ccccc}
          i_1&i_2 &\cdots &i_m &i_{m+1}\\
          j_1&j_2&\cdots &j_m &j_{m+1}\\
        \end{array}
      \right),$$
then $$detH\left(
        \begin{array}{ccccc}
          i_1&i_2 &\cdots &i_m  &i_{m+1}\\
          j_1&j_2&\cdots &j_m &j_{m+1}\\
        \end{array}
      \right)= 0,$$
so $ \gamma_{i_1},\gamma_{i_2},\cdots,\gamma_{i_m},\gamma_{i_{m+1}}$ are linear dependent, rank$(\gamma_1,\gamma_2,\cdots,\gamma_l)=m$, and all of the $(m+1)\times(m+1) $ minors of H are zeros, therefore $x_1\mid d_{m+1}(F)$.
\end{proof}
remark: The same goes for the rest.

\begin{lemma}\cite{Liu2024} Let $F \in\mathbf{R}^{l\times m}$ ($l\leq m$) be of full row rank, $d(F)=d_1^{p}$, where $d_1=x_1-f_1(x_2,\cdots,x_n)$, $p$ is a positive integer. If $J_k(F)=\mathbf{R}$ for $k=1,2,\cdots,l$, then $F$ is equivalent to its  Smith normal form$[D ~~~ 0_{l\times(m-l)}]$.
where $D=diag\{d_1^{r_1},d_1^{r_2},\cdots,d_1^{r_l}\}$, $p_0\equiv0$, $r_k=p_k-p_{k-1}, d_k(F)=d_1^{p_k}$, $k=1,2,\cdots,l$.
\end{lemma}

\begin{lemma} Let $F,D,C \in \mathbf{R}^{l\times l}, F=D \cdot C,~ d_i(F)=x_1^{p_i}x_2^{q_i}$, $i=1,2,\cdots,m+1$, and
$$D=\left(
   \begin{array}{cccccc}
          x_1^{r_1}x_2^{s_1}&&&&&\\
          &\ddots&&&&\\
          &&x_1^{r_m}x_2^{s_m}&&&\\
          &&&x_1^{r_{m+1}}x_2^{s_{m+1}}E_{l-m}&&\\
\end{array}
 \right)
,$$
where  $r_1\leq r_2 \leq\cdots\leq r_m\leq r_{m+1},s_1\leq s_2 \leq\cdots\leq s_m\leq s_{m+1}$, $p_i=r_1+r_2+\cdots+r_i,q_i=s_1+s_2+\cdots+s_i$, $i=1,2,\cdots,m$.
\\ If $J_m(F)=\mathbf{R}$, $p_m=r_1+r_2+\cdots+r_m,q_m=s_1+s_2+\cdots+s_m$, when $p_{m+1}>r_1+r_2+\cdots+r_m+r_{m+1}$, then $d_m(C)=1$, $J_m(C)=\mathbf{R}$, $x_1\mid d_{m+1}(C)$ or when $q_{m+1}>s_1+s_2+\cdots+s_m+s_{m+1}$, then $d_m(C)=1$, $J_m(C)=\mathbf{R}$,  $x_2\mid d_{m+1}(C)$.
\end{lemma}
\begin{proof}
Since $d_m(F)=d_m(D)=x_1^{p_m}x_2^{q_m}$, by Lemma 3.3, $d_m(C)=1,J_m(C)=\mathbf{R}$.
We assumed that  $p_m=r_1+r_2+\cdots+r_m, q_m=s_1+s_2+\cdots+s_m$.\\
When $p_{m+1}>r_1+r_2+\cdots+r_m+r_{m+1}$, since\\
$$\begin{aligned}
&& detF\left(
       \begin{array}{cccccc}
         m_1   & m_2  &\cdots &m_q\\
         l_1 &l_2   &\cdots &l_q\\
       \end{array}
     \right) \end{aligned} 
 \begin{aligned} =x_1^{r_{m_1}+\cdots+r_{m_q}}\cdot x_2^{s_{m_1}+\cdots+s_{m_q}}
     \cdot detC\left(
       \begin{array}{cccccc}
         m_1    &\cdots &m_q\\
         l_1    &\cdots &l_q\\
       \end{array}
     \right).
       \end{aligned} (2)$$

If $r_1=r_2=\cdots=r_{m+1} $, since $d_m(C)=1$, then there are $i_1,i_2,\cdots,i_m$; $j_1,j_2,\cdots,j_m$ such that $$x_1\nmid detC\left(
       \begin{array}{cccccccc}
         i_1 ~  & i_2  & \cdots  \cdots &i_m\\
         j_1 ~&j_2 & \cdots   \cdots &j_m\\
       \end{array}
     \right).$$
For any $k_{m+1}$, where $k_{m+1}\neq i_1,i_2,\cdots,i_m$, combined with $r_1=r_2=\cdots=r_{m+1}, ~ r_{i_1}+r_{i_2}+\cdots+r_{i_m}+r_{k_{m+1}}=r_1+r_2+\cdots+r_{m+1}$, we have that
$$\begin{aligned}
     &detF\left(
       \begin{array}{ccccccccc}
         i_1~& i_2&\cdots &~i_m&~k_{m+1} \\
         h_1~&h_2& \cdots&~h_m&~h_{m+1} \\
       \end{array}
     \right) \end{aligned}$$
    $$\begin{aligned} = x_1^{r_1+\cdots+r_{m+1}}\cdot x_2^{s_1+\cdots+s_{m+1}}
    \cdot detC\left(
       \begin{array}{cccccccc}
         i_1 ~ &\cdots&k_{m+1} \\
         h_1~& \cdots&h_{m+1} \\
       \end{array}
     \right),
     \end{aligned}$$
where $h_1,h_2,\cdots, h_{m+1}$ is any permutation of $1,2,\cdots,l$. Since $d_{m+1}(F)=x_1^{p_{m+1}}x_2^{q_{m+1}}$, $p_{m+1}>r_1+r_2+\cdots+r_m+r_{m+1}$, we have
 $$x_1\mid detC\left(
       \begin{array}{ccccccccc}
         i_1~& i_2&\cdots &i_m~&k_{m+1} \\
         h_1~&h_2& \cdots&h_m~&h_{m+1} \\
       \end{array}
     \right).$$
By Lemma  3.6, $x_1\mid d_{m+1}(C)$.

If there is $t \ (1\leq t\leq m)$ such that $r_t<r_{t+1}= r_{t+2}= \cdots= r_{m+1}$, Since $d_m(F)=x_1^{r_1+r_2+\cdots+r_m}x_2^{s_1+s_2+\cdots+s_m}$, from (2), there are $i_{t+1},\cdots,i_m$ and $h_1,h_2,\cdots,h_m$ such that
 $$x_1\nmid detC\left(
                \begin{array}{ccccccc}
                  1 ~& 2 & \cdots & t & i_{t+1} &\cdots & i_m\\
                  h_1~ & h_2 & \cdots & h_t &\  h_{t+1} &\cdots & h_m\\
                \end{array}
              \right).
 $$
Otherwise, $p_m\geq r_1+\cdots+r_m+1$, this is a contradiction. For any $k_{m+1}(k_{m+1}>t,\ k_{m+1}\neq i_{t+1},\cdots,i_m)$, any permutation $j_1,\cdots,j_m,j_{m+1}$.
Since $r_{t+1}=r_{t+2}=\cdots=r_{m+1}$, we have
 $$ \begin{aligned}
     &detF\left(
       \begin{array}{ccccccccc}
         1\ \ \cdots \ \ t\ \ i_{t+1}\ \  \cdots \ \ k_{m+1} \\
         j_1 \ \ \cdots\ \ j_t\ \ j_{t+1}\ \  \cdots \ \ j_{m+1} \\
       \end{array}
     \right)=\\&x_1^{r_1+\cdots+r_{m+1}}\cdot x_2^{s_1+\cdots+s_{m+1}}\cdot detC\left(
       \begin{array}{cccccccc}
         1\  \cdots \  t\ i_{t+1}  \cdots\ \ i_m \ k_{m+1} \\
         j_1\  \cdots \ j_t\  j_{t+1} \ \cdots\  j_m\  j_{m+1} \\
       \end{array}
     \right).
     \end{aligned}$$
Since $d_{m+1}(F)=x_1^{p_{m+1}}x_2^{q_{m+1}}$, we have
 $$x_1\mid detC\left(
       \begin{array}{ccccccccc}
         1&\cdots & t~&i_{t+1}& \cdots&i_m &k_{m+1} \\
         j_1& \cdots~&j_t~&j_{t+1}& \cdots&j_m &j_{m+1} \\
       \end{array}
     \right).$$
By Lemma 3.6 , $x_1\mid d_{m+1}(C)$. 
when $q_{m+1}>s_1+s_2+\cdots+s_m+s_{m+1}$, the proof is same of preceding process, thus $x_2\mid d_{m+1}(C)$.
\end{proof}

\begin{lemma} 
Let $F\in \mathbf{R}^{l\times l}$, $d_i(F)=x_2^{s_1+\cdots +s_i}$,  $J_i(F)=\mathbf{R}$, $i=1,2,...,l$ and $$B=\left(
   \begin{array}{cc}
          E_t&\\
          & x_1E_{l-t}\\
 \end{array}
 \right) \cdot V \cdot\left(
   \begin{array}{cccc}
          x_2^{s_1} &&&\\
          &\ddots &&\\
          & &x_2^{s_t}&\\
          &&& x_2^{s}E_{l-t}\\
 \end{array}
 \right)
.$$ where $s_1\leq \cdots \leq s_t\leq s$, $V\in \overline{\mathbf{R}}_1^{l\times l},$
if $F\sim B\cdot G_1$, then we have

(1) $d_i(B)=x_2^{s_1+\cdots +s_i}$, $J_i(B)=\mathbf{R}$, $i=1,2,...,t.$

(2) $F\sim D\cdot G_1$, where $$D=\left(
   \begin{array}{cccccc}
        x_2^{s_1} &&&\\
          &\ddots &&\\
          & &x_2^{s_t}&\\
          &&& x_1x_2^{s}E_{l-t}\\
          \end{array}
 \right)$$ and $G\in \mathbf{R}^{l\times l}$.
\end{lemma}
\begin{proof}
(1) Since $d_i(B\cdot G_1)=d_i(F)$,then $d_i(B)\mid d_i(F)$, and $d_i(F)=x_2^{s_1+\cdots +s_i}$, we have $x_2^{s_1+\cdots +s_i}\mid d_i(B)$, it can be seen that $d_i(B)=x_2^{s_1+\cdots +s_i}=d_i(F)$, by lemma3.3, $J_i(B)=J_i(F)=\mathbf{R},i=1,2,...,t.$

(2) Let $$ V=\left(
   \begin{array}{ccc}
          v_1& v_2\\
          v_3& v_4\\
\end{array}
 \right),$$
where $v_1\in \overline{\mathbf{R}}_1^{t\times t}$, $v_2\in \overline{\mathbf{R}}_1^{t\times (l-t)}$, $v_3\in \overline{\mathbf{R}}_1^{(l-t)\times t}$ , $v_4\in \overline{\mathbf{R}}_1^{(l-t)\times (l-t)}$, therefore $$B=\left(
   \begin{array}{ccc}
          v_1D_1& x_2^sv_2\\
          x_1v_3D_1& x_1x_2^sv_4\\
\end{array}
 \right),$$ where $D_1=\left(
   \begin{array}{cccc}
        x_2^{s_1} &&\\
          &\ddots &\\
          & &x_2^{s_t}\\
          \end{array}
 \right)$. 
 
 Let $A=(v_1D_1,\ \  x_2^sv_2)$, for $i\leq t$, let's set all $i$ minor of $A$as $\alpha_1,..., \alpha_u$, obviously, it is also an $i$ minor of $B$, and all other $i$ minor of $B$ are $\beta_1,..., \beta_v$, so $x_1\mid \beta_j,\ j=1,2,...,v.$ Since $d_i(B)=x_2^ {s_1+s_2+\cdots+s_i}$, then $x_1\mid \beta_j/d_i(B)$, we have $\beta_j/d_i(B)=x_1\cdot w_j, j=1,2,...,v.$ Because $J_i(B)=\mathbf{R}$, there exists $f_k,\ g_j\in \mathbf{R}$ such that $$\sum_{k=1}^uf_k\cdot \alpha_k/d_i(B)+\sum_{j=1}^vg_j\cdot \beta_j/d_i(B)=1$$ then we have  $$\sum_{k=1}^uf_k\cdot \alpha_k/d_i(B)+\sum_{j=1}^vg_j\cdot x_1w_j=1$$ Since $v_1,v_2,D_1,d_i(B)$ all without variables $x_1$, so $\alpha_k$  without variables $x_1$, In the above equation, let $x_1=0$, then $$\sum_{k=1}^uf_k(0,x_2,...,x_n)\cdot \alpha_k/d_i(B)=1$$ From the above equation $d_i(A)=d_i(B)=x_2^{s_1+\cdots+s_i}$, and $J_i(A)=\mathbf{R}, i=1,...,t$, From Theorem 2.5 in \cite{ref17}, therefore, there exists a unimodular matrix $P_1\in \mathbf{R}^{t\times t}$ and $Q_1\in \mathbf{R}^{l\times l}$ such that
$$P_1\cdot A\cdot Q_1=(D_1\ \ 0)$$
Let $$P=\left(
   \begin{array}{cc}
       P_1 &  \\
          & E_{l-t} \\
          \end{array}
 \right),Q_1=\left(
   \begin{array}{cc}
       Q_{11} &Q_{12}  \\
       Q_{21}   & Q_{22} \\
          \end{array}
 \right)$$where $Q_{11}\in \overline{\mathbf{R}}_1^{t\times t}$, $Q_{12}\in \overline{\mathbf{R}}_1^{t\times (l-t)}$, $Q_{21}\in \overline{\mathbf{R}}_1^{(l-t)\times t}$ , $Q_{22}\in \overline{\mathbf{R}}_1^{(l-t)\times (l-t)}$, then $$P\cdot B\cdot Q_1=\left(
   \begin{array}{cc}
       D_1 & 0  \\
       x_1v_3D_1Q_{11}+x_1x_2^sv_4Q_{21} & x_1v_4^{'} \\
          \end{array}
 \right)$$
where $v_4^{'}=v_3D_1Q_{12}+x_2^sv_4Q_{22}.$ Since $s_1\leq \cdots \leq s_t\leq s$, So from elementary transformations, we have
 $$B\sim C=\left(
   \begin{array}{cc}
       D_1 &0  \\
       0   & x_1v_4^{'} \\
          \end{array}
 \right).$$
 let $v_4^{'}=(b_{ij})_{(l-t)\times (l-t)}$, since $d_{t+1}(B)=x_1x_2^{s_1+\cdots +s_t+s}$, and the $t+1$ minor of $C$ is $x_1det D_1\cdot b_{ij}$, $d_{t+1}(C)=d_{t+1}(B)=x_1x_2^{s_1+\cdots +s_t+s}=x_1detD_1x_2^s$, from $d_{t+1}(C)\mid x_1detDb_{ij}$, we have $x_2^s\mid b_{ij}$. Let $v_4^{''}=v_4^{'}/x_2^s$, we can get  
 $$C\sim C_1=\left(
   \begin{array}{cc}
       D_1 &0  \\
       0   & x_1x_2^sv_4^{''} \\
          \end{array}
 \right)=\left(
   \begin{array}{cc}
       D_1 &0  \\
       0   & x_1x_2^sE_{l-t} \\
          \end{array}
 \right)\cdot \left(
   \begin{array}{cc}
       E_t &0  \\
       0   & v_4^{''} \\
          \end{array}
 \right)$$
Since $detC_1=detC=detB=detD_1\cdot (x_1x_2^s)^(l-t)$, then $det\left(
   \begin{array}{cc}
       E_t &0  \\
       0   & v_4^{''} \\
          \end{array}
 \right)=1$ therefore $B\sim C_1\sim D$, $F\sim D\cdot G.$
\end{proof}

If $x_1^k\mid d_1(F)$, $x_1^{k+1}\nmid d_1(F)$, We can consider $F^{'}=F/x_1^k$, Obviously, studying the equivalence between $F$ and its Smith form is equivalent to studying the equivalence between $F^{'}$ and its Smith form, therefore, our general assumption is as follows $x_1\nmid d_1(F)$, $x_2\nmid d_1(F)$, and $x_1\cdot x_2\mid d(F)$, i.e. $d(F)=x_1^px_2^q$, where $p\neq 0,q\neq0.$

\begin{lemma}
Let $$B=\small\left(
   \begin{array}{cccc}
          x_1^{r_1}x_2^{s_1} &&&\\
          &\ddots &&\\
          & &x_1^{r_k}x_2^{s_k}&\\
          &&& x_1^{r}x_2^{s}E_{l-k}\\
 \end{array}
 \right) \cdot V \cdot\left(
   \begin{array}{cc}
          E_k&\\
          & x_1E_{l-k}\\
 \end{array}
 \right)
.$$
If $J_i(B)=\mathbf{R}$, $d_i(B)=x_1^{r_1+\cdots+r_i}x_2^{s_1+\cdots+s_i},i=1,2,...,k.$$V\in \overline{\mathbf{R}}_1^{l\times l}$,$r_1\leq r_2\leq \cdots \leq r_k\leq r$, $s_1\leq s_2\leq \cdots \leq s_k\leq s$,
      then  $$B\sim B_1=\left(
   \begin{array}{cccc}
         x_1^{r_1}x_2^{s_1} &&&\\
          &\ddots &&\\
          & &x_1^{r_k}x_2^{s_k}&\\
          &&&x_1^{r+1}x_2^{s}E_{l-k} \\
          \end{array}
 \right).$$
\end{lemma}

\begin{proof}
When $r=0,s=0$, the conclusion is clearly valid. When$r=0, s>0$, by lemma 3.8, consider $B^T$, the conclusion is clearly valid, therefore, consider the case where $r>0$, from $x_1\nmid d_1(F)$, $x_2\nmid d_1(F)$, we have $r_1,r_2,...,r_k, r$ not completely equal, so there exist $t<k$ such that $r_1\leq r_2\leq\cdots \leq r_t<r_{t+1}=\cdots=r_k=r.$
Let $$ V=\left(
   \begin{array}{ccc}
          v_1& v_2\\
          v_3& v_4\\
\end{array}
 \right),$$
where $v_1\in \overline{\mathbf{R}}_1^{t\times k}$, $v_2\in \overline{\mathbf{R}}_1^{t\times {(l-k)}}$, $v_3\in \overline{\mathbf{R}}_1^{{(l-t)}\times k}$, $v_4\in \overline{\mathbf{R}}_1^{{(l-t)}\times {(l-k)}}.$
$$B=\left(
   \begin{array}{cccc}
          x_1^{r_1}x_2^{s_1} &&&\\
          &\ddots &&\\
          & &x_1^{r_k}x_2^{s_k}&\\
          &&& x_1^{r}x_2^{s}E_{l-k}\\
 \end{array}
 \right) \cdot \left(
   \begin{array}{ccc}
          v_1& x_1v_2\\
          v_3& x_1v_4\\
\end{array}
 \right)
.$$

Certificate $(v_1~ x_1v_2)$ is a unimodular row.
Let $t$ minor of $(v_1~ v_2)$ are $\alpha_1,...,\alpha_h$(Obviously, there is no common zero point), then the $t$ minor of $(v_1~ x_1v_2)$ are 
$\alpha_1,...,\alpha_{\beta},x_1^{m_2}\alpha_{\beta+1},...,x_1^{m_s}\alpha_h$, if $(v_1~ x_1v_2)$ is not a unimodular row, thne $\alpha_1,...,\alpha_{\beta},x_1$ have a common zero point $p_0=(0,x_{20},x_{30}...,x_{n0})$, consider $t$ reduced minor of $B$, It can be divided into two categories, one of which is a $t$ reduced minor of matrix $diag\{x_1^{r_1}x_2^{s_1},...,x_1^{r_t}x_2^{s_t}\}\cdot v_1$: i.e. $\alpha_1,...,\alpha_{\beta};$ The other type is other reduced minor, but they all contain $x_1$ as their factor. Therefore, all $t$ reduced minor of $B$ have a common zero point $p_0$, This contradicts $J_t(B)=\mathbf{R}$.

By Quillen-Suslin theorem, there exist a unimodular matrix  $Q\in \overline{\mathbf{R}}_1^{l\times l}$ such that
$(v_1~x_1v_2)\cdot Q=(E_t~0)$, so
$$ B\cdot Q=\left(
   \begin{array}{ccccccccc}
          x_1^{r_1}x_2^{s_1} &&&&&&\\
          &\ddots &&&&&\\
          & &x_1^{r_t}x_2^{s_t}&&&&\\
          x_1^{r}x_2^{s_{t+1}}v_{t+1,1}^{'}&\cdots&x_1^{r}x_2^{s_{t+1}}v_{t+1,t}^{'}&x_1^{r}x_2^{s_{t+1}}v_{t+1,t+1}^{'}&\cdots&x_1^{r}x_2^{s_{t+1}}v_{t+1,l}^{'}&\\
          \cdots&\cdots&\cdots&\cdots&\cdots&\cdots&\\
          x_1^{r}x_2^{s_k}v_{k,1}^{'}&\cdots&x_1^{r}x_2^{s_k}v_{k,t}^{'}&x_1^{r}x_2^{s_k}v_{k,t+1}^{'}&\cdots&x_1^{r}x_2^{s_k}v_{t+1,l}^{'}&\\
         \cdots&\cdots&\cdots&\cdots&\cdots&\cdots&\\
         x_1^{r}x_2^{s}v_{l,1}^{'}&\cdots&x_1^{r}x_2^{s}v_{l,t}^{'}&x_1^{r}x_2^{s}v_{l,t+1}^{'}&\cdots&x_1^{r}x_2^{s}v_{l,l}^{'}&\\
 \end{array}
 \right),
 $$
 By elementary line transformation, we can obtain $$ B\sim C=\left(
   \begin{array}{ccccccccc}
          x_1^{r_1}x_2^{s_1} &&&&&&\\
          &\ddots &&&&&\\
          & &x_1^{r_t}x_2^{s_t}&&&&\\
          0&\cdots&0&x_1^{r}x_2^{s_{t+1}}v_{t+1,t+1}^{'}&\cdots&x_1^{r}x_2^{s_{t+1}}v_{t+1,l}^{'}&\\
          \cdots&\cdots&\cdots&\cdots&\cdots&\cdots&\\
          0&\cdots&0&x_1^{r}x_2^{s_k}v_{k,t+1}^{'}&\cdots&x_1^{r}x_2^{s_k}v_{t+1,l}^{'}&\\
         \cdots&\cdots&\cdots&\cdots&\cdots&\cdots&\\
         0&\cdots&0&x_1^{r}x_2^{s}v_{l,t+1}^{'}&\cdots&x_1^{r}x_2^{s}v_{l,l}^{'}&\\
 \end{array}
 \right),
 $$
Let $ C=\left(
   \begin{array}{cccc}
         x_1^{r_1}x_2^{s_1} &&&\\
          &\ddots &&\\
          & &x_1^{r_t}x_2^{s_t}&\\
          &&&M\cdot V^{'} \\
          \end{array}
 \right)$,where $M=diag\{x_1^rx_2^{s_{t+1}},...,x_1^rx_2^{s_k},x_1^rx_2^s,...,x_1^rx_2^s\}$. Prove the existence of unimodular matrix $P$ such that $$V^{'}=P\cdot \left(
   \begin{array}{cccc}
         E_{k-t} &\\
          & x_1E_{l-k} \\
          \end{array}
 \right)\cdot G$$, where $G\in \mathbf{R}^{{(l-t)}\times (l-t)}.$
 
Let $e=\sum_{i=1}^t r_i+(k-t+1)\cdot r+1$, Since $x_1^e\mid d_{k+1}(B)$, and $B\sim C$, then $x_1^e\mid d_{k+1}(C)$, 
 Let $V^{''}$is $(k-t+1)$ submatrix of $V^{'}$, $N$ is any $(k-t+1) $ submatrix of $M$, visibility $detN=x_1^{(k-t+1)\cdot r}\cdot x_2^h$, let $$C^{'}=\left(
   \begin{array}{cccc}
         x_1^{r_1}x_2^{s_1} &&&\\
          &\ddots &&\\
          & &x_1^{r_t}x_2^{s_t}&\\
          &&&N\cdot V^{''} \\
          \end{array}
 \right)$$ then $detC^{'}$ is $(k+1)$ minor of $C$. So $x_1^e\mid detC^{'}$, since $detC^{'}=x_1^{r_1+\cdots +r_t}x_2^{s_1+\cdots +s_t}\cdot detN \cdot detV^{''}$, $detN= x_1^{(k-t+1)\cdot r}x_2^h$, we have $$(x_1^e=x_1^{r_1+\cdots +r_t+(k-t+1)\cdot r+1})\mid (detC^{'}=x_1^{r_1+\cdots +r_t+(k-t+1)\cdot r}x_2^{s_1+\cdots +s_t+h}\cdot detV^{''})$$
From the above equation, it can be concluded that $x_1\mid detV^{''}$. Arbitrariness of $V^{''}$, so $x_1\mid d_{k-t+1}(V^{'})$, because $J_i(B)=\mathbf{R}$, $d_i(B)=x_1^{r_1+\cdots+r_i}x_2^{s_1+\cdots+s_i},i=1,2,...,k.$, we have $J_{k-t}(V^{'})=1$, then we prove that $x_1\nmid d_{k-t}(V^{'})$, if $x_1\mid d_{k-t}(V^{'})$, then $x_1^{(r_1+\cdots +r_k+1)}\mid (d_k(C)=d_k(B)=x_1^{r_1+\cdots+r_k}x_2^{s_1+\cdots+s_k})$, Obviously contradictory , so by lemma 3.5, we have $$V^{'}=P\cdot \left(
   \begin{array}{cccccc}
    E_{k-t}&\\
    & x_1E_{l-k}\\
 \end{array}
 \right)\cdot G,$$ 
 Substitute to obtain $$M\cdot V^{'}=\left(
   \begin{array}{cccc}
          x_1^rx_2^{s_{t+1}} &&&\\
          &\ddots &&\\
          & &x_1^rx_2^{s_k}&\\
          &&& x_1^rx_2^{s}E_{l-k}\\
 \end{array}
 \right)\cdot P\cdot \left(
   \begin{array}{cccccc}
    E_{k-t}&\\
    & x_1E_{l-k}\\
 \end{array}
 \right)\cdot G$$Since $\left(
   \begin{array}{cccc}
          x_1^rx_2^{s_{t+1}} &&&\\
          &\ddots &&\\
          & &x_1^rx_2^{s_k}&\\
          &&& x_1^rx_2^{s}E_{l-k}\\
 \end{array}
 \right)=x_1^r\cdot \left(
   \begin{array}{cccc}
          x_2^{s_{t+1}} &&&\\
          &\ddots &&\\
          & &x_2^{s_k}&\\
          &&& x_2^{s}E_{l-k}\\
 \end{array}
 \right)$, by lemma3.8, we can get 
 $$M\cdot V^{'}= x_1^r\cdot P^{'}\cdot \left(
   \begin{array}{cccc}
          x_2^{s_{t+1}} &&&\\
          &\ddots &&\\
          & &x_2^{s_k}&\\
          &&& x_1x_2^{s}E_{l-k}\\
 \end{array}
 \right)\cdot G{'}=P^{'}\cdot \left(
   \begin{array}{cccc}
          x_1^rx_2^{s_{t+1}} &&&\\
          &\ddots &&\\
          & &x_1^rx_2^{s_k}&\\
          &&& x_1^{r+1}x_2^{s}E_{l-k}\\
 \end{array}
 \right)\cdot G{'}$$
From $$C=\left(
   \begin{array}{cccc}
         x_1^{r_1}x_2^{s_1} &&&\\
          &\ddots &&\\
          & &x_1^{r_t}x_2^{s_t}&\\
          &&&M\cdot V^{'} \\
          \end{array}
 \right)=\left(
   \begin{array}{cccccc}
   E_{t}&\\
    & P^{'}\\
 \end{array}
 \right)\cdot \left(
   \begin{array}{cccc}
         x_1^{r_1}x_2^{s_1} &&&\\
          &\ddots &&\\
          & &x_1^{r_k}x_2^{s_k}&\\
          &&&x_1^{r+1}x_2^{s}E_{l-k} \\
          \end{array}
 \right)\cdot \left(
   \begin{array}{cccccc}
   E_{t}&\\
    & G^{'}\\
 \end{array}
 \right)$$
 By verifying the determinant, the determinant of $P,G,P^{'},G^{'}$ is a non-zero constant, so it is reversible, then $$B\sim C\sim \left(
   \begin{array}{cccc}
         x_1^{r_1}x_2^{s_1} &&&\\
          &\ddots &&\\
          & &x_1^{r_k}x_2^{s_k}&\\
          &&&x_1^{r+1}x_2^{s}E_{l-k} \\
          \end{array}
 \right).$$
 
 When $r_1\leq r_2\leq\cdots \leq r_k<r$, same as the above proof.
 
\end{proof}

\begin{lemma}
Let $$B=\left(
   \begin{array}{cccc}
          x_1^{r_1}x_2^{s_1} &&&\\
          &\ddots &&\\
          & &x_1^{r_t}x_2^{s_t}&\\
          &&& x_1^{r}x_2^{s}E_{l-t}\\
 \end{array}
 \right) \cdot U \cdot\left(
   \begin{array}{cc}
          E_t&\\
          & x_2E_{l-t}\\
 \end{array}
 \right)
.$$
If $J_t(B)=\mathbf{R}$, $d_t(B)=x_1^{r_1+\cdots+r_t}x_2^{s_1+\cdots+s_t}$, $U\in \overline{\mathbf{R}}_2^{l\times l}$, $r_1\leq r_2\leq \cdots \leq r_t\leq r,s_1\leq s_2\leq \cdots \leq s_t\leq s$;
      then  $$B\sim B_1=\left(
   \begin{array}{cccc}
         x_1^{r_1}x_2^{s_1} &&&\\
          &\ddots &&\\
          & &x_1^{r_t}x_2^{s_t}&\\
          &&&x_1^{r}x_2^{s+1}E_{l-t} \\
          \end{array}
 \right).$$

\end{lemma}

\begin{proof}
The proof is similar to the proof of Lemma 3.10.
\end{proof}

\begin{lemma}(\cite{Liu2014}) Let $F\in \mathbf{R}^{l\times m}(l\leq m)$ be of full row rank. If $J_l(F)=\mathbf{R}$, then $F$ admits a zero prime factorization $F=G\cdot F_1$ with $G\in \mathbf{R}^{l\times l}$, $F_1\in \mathbf{R}^{l\times m}$, and $det G=d_l(F)$, $F_1$ is a ZLP matrix.
\end{lemma}
\begin{lemma}(\cite{Wang2004}) Let $F\in \mathbf{R}^{l\times m}$  be of normal rank $r$. If $J_r(F)=\mathbf{R}$, then $F$ can be decomposed into $F=G\cdot F_1$ such that $F_1 \in \mathbf{R}^{r\times m}$ is ZLP, $G\in R^{l\times r}$ and $J_r(G)=R$.
\end{lemma}

We consider the following map:

\begin{center}
    
$\phi: K[x_1,\cdots,x_n]\mapsto  K[x_1,\cdots,x_n]; $\\
      $~~x_1 \mapsto x_1-f_1(x_2,\cdots,x_n),$\\
          $x_2 \mapsto x_2-f_2(x_3,\cdots,x_n),$\\
          $ x_3 \mapsto  x_3,$\\
   $\cdots$ \\
   $ x_n \mapsto  x_n.$
   \end{center}
\begin{lemma} The homorphism $\phi: K[x_1,\cdots,x_n]\mapsto  K[x_1,\cdots,x_n] $ is automorphism.
\end{lemma}

\begin{proof}
since $\phi$ is an homomorphism;suppose $f=ax_1^{r_1}x_2^{r_2}\ldots x_n^{r_n}+\cdots$. We define the lexicographical order on $K[x_1,\cdots,x_n] $, 
and $\phi(f)=a(x_1-f_1)^{r_1}(x_2-f_2)^{r_2}\ldots x_n^{r_n}+\cdots$. Therefor $\phi(f)\neq 0$, and $\phi$ is an  injective homomorphism;
 then we prove  $\phi$ is an surjective homomorphism: $x_3,\ldots,x_n \in Im(\phi)$, since $ f_2\in Im(\phi)$, and $x_2-f_2\in Im(\phi)$;
$ x_2\in Im(\phi)$. similarly, $ x_1\in Im(\phi)$. thus The homorphism $\phi: K[x_1,\cdots,x_n]\mapsto  K[x_1,\cdots,x_n] $ is automorphism.
the proof is complete.
\end{proof}

\section{Results }

\begin{theorem} Let $F\in \mathbf{R}^{3\times 3}$, $d(F)\triangleq d_l(F)=x_1^{p}x_2^{q}$, where $p,q$ is positive integer. If $J_i(F)=\mathbf{R}$, for $i=1,2,3$, then $F$ is equivalent to its Smith normal form $D$, \\
 
 $$D=\left(
   \begin{array}{ccc}
           x_1^{r_1}x_2^{s_1}& &\\
            & x_1^{r_2}x_2^{s_2}& \\
                      & &  x_1^{r_3}x_2^{s_3}\\
            \end{array}
 \right),$$
where$r_1\leq r_2\leq r_3,s_1\leq s_2\leq s_3.$
\end{theorem}
\begin{proof}
For $d(F)=x_1^{p}x_2^{q}$, since $d_i(F)\mid d(F)$, we may assume that $d_i(F)=x_1^{p_i}x_2^{q_i}$, where $p_i,q_i$ is a non-negative integer, $i=1,2,3$. Set $p_0=q_0\equiv 0$, $r_i=p_i-p_{i-1},s_i=q_i-q_{i-1}$, $i=1,2,3$. And there exist $k= 1,2,3$ such that $x_1x_2\mid d_{k}(F)$, let $t= min (k)$;
($t=0$ is same as lemma 3.8).

 (1) if $t\leq2$

 When $t=1$, $d_1(F)=x_1^{p_1}x_2^{q_1}=x_1^{r_1}x_2^{s_1}$, then we have $F=x_1^{r_1}x_2^{s_1}\cdot N_1$.
 
 When $t=2$, $d_1(F)=x_1^{p_1}=x_1^{r_1}$ or $d_1(F)=x_2^{q_1}=x_2^{s_1}$, then we have $F=x_1^{r_1}x_2^{s_1}\cdot N_1$, where $r_1=0$ or $s_1=0$.

  It is obvious that $p_1=r_1, q_1=s_1,p_2\geq 2r_1,q_2\geq 2s_1$, $r_2=p_2-p_1\geq r_1,s_2=q_2-q_1\geq s_1$. Setting $m_1= min(r_2-r_1,s_2-s_1)$,

 If $m_1=0$, we assumed $r_2=r_1,s_2>s_1$, since $J_i(F)=\mathbf{R}$, by Lemma 3.1,$J_1(N_1)=\mathbf{R}$, and $x_2\nmid d_1(N_1),\ x_2\mid d_2(N_1)$, by Lemma 3.5, there exist  unimodular matrix $U_1\in \overline{\mathbf{R}}_2$ such that
 $$
 U_1 N_1=\left(
   \begin{array}{ccc}
        1 &  &\\
          & x_2 &\\
          & &  x_2\\
          \end{array}
 \right)\cdot Q^{'},$$
 i.e.
 $$F \sim B = \left(
   \begin{array}{cc}
        x_1^{r_1}x_2^{s_1} &  \\
          & x_1^{r_1}x_2^{s_1+1}E_2 \\
          \end{array}
 \right)\cdot Q_1.$$

 If $m_1>0$, so $r_2>r_1$, and $s_2>s_1$, since $J_i(F)=\mathbf{R}$, by Lemma 3.1, $J_1(N_1)=\mathbf{R}$, and $x_i\nmid d_1(N_1), x_i\mid d_2(N_1)$, by Lemma 3.5, there exist  unimodular matrix $U_1\in \overline{\mathbf{R}}_2,V_1\in \overline{\mathbf{R}}_1$ such that
 $$
 U_1 N_1=\left(
   \begin{array}{ccc}
        1 &  &\\
          & x_2 &\\
          & &  x_2\\
          \end{array}
 \right)\cdot V_1\cdot \left(
   \begin{array}{ccc}
         1 &  &\\
          & x_1 &\\
          & &  x_1\\
          \end{array}
 \right)\cdot Q^{'},$$
by lemma 3.10,
$$N_1\sim \left(
   \begin{array}{ccc}
        1 &  &\\
          & x_1x_2  & \\
          & &  x_1x_2 \\
          \end{array}
 \right)\cdot Q_1,$$
i.e.
 $$F\sim B = \left(
   \begin{array}{ccc}
        x_1^{r_1}x_2^{s_1} &  &\\
          & x_1^{r_1+1}x_2^{s_1+1} &\\
          &  &  x_1^{r_1+1}x_2^{s_1+1}
          \end{array}
 \right)\cdot Q_1.$$

 By Lemma 3.3, $d_1(B)=x_1^{r_1}x_2^{s_1}$,$J_1(B)=\mathbf{R}$, combined with Lemma 3.8, $d_1(Q_1)=1$,$J_1(Q_1)=\mathbf{R}$, $x_i\mid d_2(Q_1)$, thus iterating the same procedure,
 $$Q_1\sim \left(
   \begin{array}{ccc}
        1 &  &\\
          & x_i  & \\
          &  &  x_i\\
          \end{array}
 \right)\cdot Q_1^{'},\ i=1,2.$$

Iterating the same procedure, combined with Lemma 3.11 and 3.12, we have $F$ is equivalent to $M_2$, where
$$
M_2=\left(
   \begin{array}{ccc}
                                   x_1^{r_1}x_2^{s_1} & & \\
                                    & x_1^{r_2}x_2^{s_2}  & \\
                                    &  &  x_1^{r_2}x_2^{s_2}\\
                                 \end{array}
 \right)\cdot F_1.
$$
Iterating the preceding process, we have$F$ is equivalent to its Smith normal form $D$;\\

(2) if $t=3$
 
Assumed $x_1\nmid d_2(F),x_1\mid d_{3}(F)$, it implied that $r_1=r_2=0$, $d_1(F)=x_2^{q_1}=x_2^{s_1}$, then we have $F=x_2^{s_1}\cdot N_1$, and $J_1(N_1)=\mathbf{R}$, by Lemma 3.5, there exist  unimodular matrix $U_1\in \overline{\mathbf{R}}_2$, such that
 $$
 D=U_1 N_1=\left(
   \begin{array}{ccc}
         1 & & \\
        & x_2 &\\
         & &  x_2 \\
          \end{array}
 \right)\cdot  Q_1,$$
 and $x_2\nmid d_1(Q_1)$,$x_2^{s_2-s_1-1}\mid d_2(Q_1)$,$J_1(Q_1)=\mathbf{R}$, by Lemma 3.5, there exist  unimodular matrix$U_2 \in \overline{\mathbf{R}}_2$ such that
 $$
 U_2 Q_1=\left(
   \begin{array}{ccc}
        1 & & \\
        & x_2 &\\
         & &  x_2 \\
          \end{array}
 \right)\cdot  Q_2,$$

so by Lemma 3.12,
$$N_1\sim \left(
   \begin{array}{ccc}
    1 & & \\
        & x_2^2 &\\
         & &  x_2^2 \\
                  \end{array}
 \right)\cdot G_1.$$

Then repeating the process ,we have
$$F\sim F_1=\left(
   \begin{array}{ccc}
    x_2^{s_1} & & \\
       & x_2^{s_2}& \\
        & &  x_2^{s_2} \\
          \end{array}
 \right)\cdot N_2,$$

Since $x_1\nmid d_2(F), x_1\mid d_3(F)$ and $F\sim F_1$,$J_{2}(F)=\mathbf{R}$, by Lemma 3.8 then we claim that, $x_1\nmid d_{2}(N_2),x_1\mid d_3(N_2)$, $J_{2}(N_2)=\mathbf{R}$.
By Lemma 3.5, there exist  unimodular matrix $V \in \overline{\mathbf{R}}_1$ such that
 $$
 V N_2=\left(
   \begin{array}{ccc}
        1 & &  \\
        & 1 &\\
         &  & x_1 \\
          \end{array}
 \right)\cdot  G_1,$$
so
$$ F_1= B\cdot G_1=\left(
   \begin{array}{ccc}
    x_2^{s_1} & & \\
       &  x_2^{s_{2}}& \\
       &  &  x_2^{s_{2}} \\
          \end{array}
 \right)\cdot V^{-1} \cdot \left(
   \begin{array}{ccc}
        1 & &  \\
        & 1 &\\
         &  & x_1 \\
          \end{array}
 \right)\cdot  G_1,$$
it is obviously that $J_i(B)=\mathbf{R}$, $d_i(B)=d_{i}(F_1)(i=1,2)$. Hence, by Lemma 3.13

 $$B \sim \left(
    \begin{array}{ccc}
    x_2^{s_1} & & \\
       & x_2^{s_{2}}& \\
         & & x_1x_2^{s_2} \\
          \end{array}
 \right).$$
 Thus
 $$ F_1 \sim \left(
    \begin{array}{ccc}
    x_2^{s_1} & & \\
       & x_2^{s_2}& \\
        & &  x_1x_2^{s_2} \\
          \end{array}
 \right)\cdot  G,$$

 where $J_i(G)=\mathbf{R}$, $d_i(G)=1(i=1,2)$, $x_i\nmid d_2(G), x_i\mid d_3(G), i=1,2$; by lemma 3.5
$$ F_1= \left(
   \begin{array}{ccc}
        x_2^{s_1} & & \\
          & x_2^{s_2}& \\
          & & x_1x_2^{s_2} \\
          \end{array}
 \right)\cdot U \cdot\left(
  \begin{array}{ccc}
        1 & & \\
        & 1 &\\
       &   &  x_i \\
          \end{array}
 \right)\cdot N^{'}
$$
iterating the same procedure successively, finally$$
 F\sim D=\left(
   \begin{array}{ccc}
        x_2^{s_1} & &\\
          & x_2^{s_2} & \\
           & & x_1^{r_3}x_2^{s_3}\\
          \end{array}
 \right).
 $$\\
Assumed $x_2\nmid d_2(F), x_2\mid d_{3}(F)$, it implied that $s_1=s_2=0$, $d_1(F)=d_1^{p_1}=d_1^{r_1}$, then we have $F=x_1^{r_1}\cdot N_1$, and $J_i(N_1)=\mathbf{R}$, we repeating the proof process, so finally$$
 F\sim D =\left(
   \begin{array}{ccc}
        x_1^{r_1} & &\\
          & x_1^{r_2}& \\
           &  & x_1^{r_3}x_2^{s_3}\\
          \end{array}
 \right).
 $$
\\
Hence $D$ is the Smith normal form of $F$.
\end{proof}

\begin{theorem} Let $F\in \mathbf{R}^{l\times l}$, $d(F)\triangleq d_l(F)=x_1^{p}x_2^{q}$, where $p,q$ is positive integer. If $J_i(F)=\mathbf{R}$, for $i=1,2,\cdots ,l$, then $F$ is equivalent to its Smith normal form $D$, \\
 $$ D=\left(
   \begin{array}{cccc}
           x_1^{r_1}x_2^{s_1}& & &\\
            & x_1^{r_2}x_2^{s_2}& &\\
            & & \ddots&\\
            & & & x_1^{r_l}x_2^{s_l}\\
            \end{array}
 \right),$$
where$r_1\leq r_2\leq\cdots\leq r_l,s_1\leq s_2\leq\cdots\leq s_l.$
\end{theorem}
\begin{proof}
For $d(F)=x_1^{p}x_2^{q}$, since $d_i(F)\mid d(F)$, we may assume that $d_i(F)=x_1^{p_i}x_2^{q_i}$, where $p_i,q_i$ is a non-negative integer, $i=1,2,\cdots,l$. Set $p_0=q_0\equiv 0$, $r_i=p_i-p_{i-1},s_i=q_i-q_{i-1}$, $i=1,2,\cdots,l$. there exist $k\in (1,\cdots,l)$ such that $x_1x_2\mid d_{k}(F)$, let $t= min (~k~)$;\\
($t=0$ is same as Lemma 3.8).
 
 (1) if $t\leq2$
 
 When $t=1$, $d_1(F)=x_1^{p_1}x_2^{q_1}=x_1^{r_1}x_2^{s_1}$, then we have $F=x_1^{r_1}x_2^{s_1}\cdot N_1$.

 When $t=2$, $d_1(F)=x_1^{p_1}=x_1^{r_1}$or$d_1(F)=x_2^{q_1}=x_2^{s_1}$, then we have $F=x_1^{r_1}x_2^{s_1}\cdot N_1$, where $r_1=0$ or $s_1=0$.

 It is obvious that $p_1=r_1, q_1=s_1,p_2\geq 2r_1,q_2\geq 2s_1, r_2=p_2-p_1\geq r_1,s_2=q_2-q_1\geq s_1$. Setting $m_1= ~Min(r_2-r_1,s_2-s_1)$.

 If $m_1=0$, there exist $j$ such that $r_j>r_{j-1}$ or $s_j>s_{j-1}$. We assumed $r_2=r_1,s_2>s_1$, since $J_i(F)=\mathbf{R}$, by Lemma 3.1, $J_1(N_1)=\mathbf{R}$, and $x_2\nmid d_1(N_1),x_2\mid d_2(N_1)$, by Lemma 3.5, there exist  unimodular matrix $U_1\in \overline{\mathbf{R}}_2$ such that
 $$
 U_1 N_1=\left(
   \begin{array}{cc}
        1 &  \\
          & x_2E_{l-1} \\
          \end{array}
 \right)\cdot Q^{'},$$
 i.e.
 $$F \sim B= \left(
   \begin{array}{cc}
        x_1^{r_1}x_2^{s_1} &  \\
          & x_1^{r_1}x_2^{s_1+1}E_{l-1} \\
          \end{array}
 \right)\cdot Q_1.$$

 If $m_1>0$, so $r_2>r_1$, and $s_2>s_1$, since $J_i(F)=\mathbf{R}$, by Lemma 3.1, $J_1(N_1)=\mathbf{R}$, and $x_i\nmid d_1(N_1), x_i\mid d_2(N_1)$, by Lemma 3.5, there exist  unimodular matrix $U_1\in \overline{\mathbf{R}}_2, V_1\in \overline{\mathbf{R}}_2$ such that
 $$
 U_1 N_1=\left(
   \begin{array}{cc}
        1 &  \\
          & x_2E_{l-1} \\
          \end{array}
 \right)\cdot V_1\cdot \left(
   \begin{array}{cc}
        1 &  \\
          & x_1E_{l-1} \\
          \end{array}
 \right)\cdot Q^{'},$$
by Lemma 3.10,
$$N_1\sim \left(
   \begin{array}{cc}
        1 &  \\
          & x_1x_2E_{l-1} \\
          \end{array}
 \right)\cdot Q_1,$$
i.e.
 $$F \sim B= \left(
   \begin{array}{cc}
        x_1^{r_1}x_2^{s_1} &  \\
          & x_1^{r_1+1}x_2^{s_1+1}E_{l-1} \\
          \end{array}
 \right)\cdot Q_1.$$

Iterating the preceding process, we have$F$ is equivalent to its Smith normal form $D$.\\

(2) if $t>2$

There $x_1\nmid d_{t-1}(F), x_1\mid d_{t}(F)$, or $x_2\nmid d_{t-1}(F),x_2\mid d_{t}(F)$, it implied that $r_1=\cdots=r_{t-1}=0$,$d_1(F)=x_2^{q_1}=x_2^{s_1}$, or $s_1=\cdots=s_{t-1}=0$, $d_1(F)=x_1^{p_1}=x_1^{r_1}$.

We assumed $x_1\nmid d_{t-1}(F), x_1\mid d_{t}(F)$, so $r_1=\cdots=r_{t-1}=0$, $d_1(F)=x_2^{q_1}=x_2^{s_1}$, then we have $F=x_2^{s_1}\cdot N_1$, and $J_1(N_1)=\mathbf{R}$, by Lemma 3.5, there exist  unimodular matrix $U_1\in \overline{\mathbf{R}}_2$, such that
 $$
 D=U_1 N_1=\left(
   \begin{array}{cc}
        1 &  \\
          & x_2E_{l-1} \\
          \end{array}
 \right)\cdot  Q_1.$$

Then repeating the process ,we have
$$F \sim F_1=\left(
   \begin{array}{cccc}
    x_2^{s_1} & && \\
      &  \ddots & & \\
       & &x_2^{s_{t-1}}& \\
        & & & x_2^{s_{t-1}}E_{l-t+1} \\
          \end{array}
 \right)\cdot N_2.$$

Since $x_1\nmid d_{t-1}(F), x_1\mid d_t(F)$ and $F\sim F_1$, $J_{t-1}(F)=\mathbf{R}$. By Lemma 3.8 then we claim that, $x_1\nmid d_{t-1}(N_2),x_1\mid d_t(N_2)$, $J_{t-1}(N_2)=\mathbf{R}$.
By Lemma 3.5, there exist  unimodular matrix $V \in \overline{\mathbf{R}}_1$ such that
 $$
 V N_2=\left(
   \begin{array}{cc}
        E_{t-1} &  \\
          & x_1E_{l-t+1} \\
          \end{array}
 \right)\cdot  G_1,$$
 and then
$$ F_1= B\cdot G_1=\left(
   \begin{array}{cccc}
    x_2^{s_1} & & &\\
      &  \ddots & & \\
       & &x_2^{s_{t-1}}& \\
       &  & & x_2^{s_{t-1}}E_{l-t+1} \\
          \end{array}
 \right)\cdot V^{-1} \cdot \left(
   \begin{array}{cc}
        E_{t-1} &  \\
          & x_1E_{l-t+1} \\
          \end{array}
 \right)\cdot  G_1.$$
It is obviously that $J_i(B)=\mathbf{R}$, $d_i(B)=d_{i}(F_1),\ i=1,\cdots,t-1$, hence, by Lemma 3.13

 $$B\sim \left(
    \begin{array}{cccc}
    x_2^{s_1} & & &\\
      &  \ddots & & \\
       & &x_2^{s_{t-1}}& \\
         & && x_1x_2^{s_{t-1}}E_{l-t+1} \\
          \end{array}
 \right).$$
 Thus
 $$ F_1\sim \left(
    \begin{array}{cccc}
    x_2^{s_1} & && \\
      &  \ddots & & \\
       & &x_2^{s_{t-1}}& \\
        & & & x_1x_2^{s_{t-1}}E_{l-t+1} \\
          \end{array}
 \right)\cdot  G,$$

 where $J_i(G)=\mathbf{R}$, $d_i(G)=1,  \ i=1,\cdots,t-1$, $x_i\nmid d_{t-1}(G), x_i\mid d_t(G),i=1,2$. By Lemma 3.5
$F_1=$$$\small  \left(
   \begin{array}{cccc}
        x_2^{s_1} & && \\
        &\ddots &&\\
          & &x_2^{s_{t-1}}& \\
          & &&x_1x_2^{s_{t-1}}E_{l-t+1} \\
          \end{array}
 \right)\cdot U \cdot\left(
  \begin{array}{cc}
        E_{t-1}& \\
          &x_iE_{l-t+1} \\
          \end{array}
 \right)\cdot N^{'}.
 $$
Iterating the same procedure successively, finally$$
 F\sim D =\left(
   \begin{array}{cccccc}
        x_2^{s_1} & && &&\\
        &\ddots &&&&\\
          & &x_2^{s_{t-1}}&&& \\
          & &&x_1^{r_t}x_2^{s_t}&& \\
           & & && \ddots &\\
           & &&&& x_1^{r_l}x_2^{s_l}\\
          \end{array}
 \right).
 $$

Hence $D$ is the Smith normal form of $F$.
\end{proof}

We now give our result of equivalent of quasi weakly polynomial matrices.
\begin{theorem} Let $F\in \mathbf{R}^{l\times l}$, $d(F)\triangleq d_l(F)=d_1^{p}d_2^{q}$, $d_1=x_1-f_1(x_2,\cdots,x_n),d_2=x_2-f_2(x_3,\cdots,x_n)$, where $p,q$ is positive integer. If $J_i(F)=\mathbf{R}$, for $i=1,2,\cdots ,l$, then $F$ is equivalent to its Smith normal form $D$, where \\
 $$
 D=\left(
   \begin{array}{cccc}
           d_1^{r_1}d_2^{s_1}& & &\\
            & d_1^{r_2}d_2^{s_2}& &\\
            & & \ddots&\\
            & & & d_1^{r_l}d_2^{s_l}\\
            \end{array}
 \right),$$
where$r_1\leq r_2\leq\cdots\leq r_l,s_1\leq s_2\leq\cdots\leq s_l.$
\end{theorem}
\begin{proof}
For matrix F defines a linear homorphism from $F^l$ to $F^l$, $\phi(F)$ is means that automorphism acts on every element of F, since $det (\phi^{-1}(F))=x_1^{r_1}x_2^{s_1}$, by theorem 4.2, there exist  unimodular matrix $U,V$ such that $U\phi^{-1}(F)V=D_1$, where $D_1$is the Smith normal form of $\phi^{-1}(F)$.

Therefor  we have $\phi(U) F \phi(V)=\phi(D_1)=D$, and $D$ is the Smith normal form of $F$.

\end{proof}

\begin{theorem} Let $F\in\mathbf{R}^{l\times m}$ ($l\leq m$) be of full row rank, $d_l(F)=d_1^{p}d_2^{q}$, where $d_1=x_1-f_1(x_2,\cdots,x_n),d_2=x_2-f_2(x_3,\cdots,x_n)$, $p,q$ is a positive integer. Then $F$ is equivalent to its  Smith normal form iff $J_k(F)=\mathbf{R}$ for $k=1,2,\cdots,l$.
\end{theorem}

\begin{proof}
\textbf{Sufficiency.} Since $d_k(F)\mid d_l(F)$, we may assume that $d_k(F)=d_1^{p_k}d_2^{q_k}, k=1,2,\cdots,l$. Let
$$B=
           [D ~~~ 0_{l\times(m-l)}]
        ,
$$
where $D=diag\{d_1^{r_1}d_2^{s_1},d_1^{r_2}d_2^{s_2},\cdots, d_1^{r_l}d_2^{s_l}\}$, $p_0\equiv0,q_0\equiv0$, $r_k=p_k-p_{k-1},s_k=q_k-q_{k-1}$, $k=1,2,\cdots,l$.

By Lemma 3.15 , there are $G\in \mathbf{R}^{l\times l}$ with $detG=d_1^pd_2^q$, and $F_1\in \mathbf{R}^{l\times m}$ such that $F=G\cdot F_1$, ~and $F_1$ is ~ZLP. ~By ~Quillen-Suslin Theorem ({\cite{Quillen1976,Suslin1976}}), ~there is a matrix $N\in \mathbf{R}^{(m-l)\times m}$ such that~~ $Q_1=[F_1^T~ ~ N^T]^T$ is a unimodular matrix, then $F(z)=G\cdot F_1$$=[ G~ ~ 0_{l\times (m-l)}]\cdot Q_1$. Since $F$ is equivalent to $
[G~  ~0_{l\times (m-l)}]$, by Lemma 3.3, $d_k(G)=d_k(F)=d_1^{p_k}d_2^{q_k}$, $J_k(G)=\mathbf{R}$ for $k=1,2,\cdots,l$. combined with  Theorem 4.3, there exist unimodular matrices $U,V\in \mathbf{R}^{l\times l}$ such that $G=U\cdot D\cdot V$. It is obviously that $V\cdot F_1$ is also ZLP, from Quillen-Suslin Theorem, there is a unimodular matrix $Q\in\mathbf{R}^{m\times m}$ such that
$V\cdot F_1\cdot Q=
[E_l~ ~ 0_{l\times(m-l)}]$. Then
\begin{align*}
F\cdot Q &=G\cdot F_1 \cdot Q=(U\cdot D\cdot V)\cdot F_1\cdot Q
=U\cdot D\cdot (V\cdot F_1\cdot Q)\\&=U\cdot D\cdot                                                                                                          [E_l~~ 0_{l\times(m-l)}]=U\cdot[D~~ 0_{l\times(m-l)}]=U \cdot B.
\end{align*}
Thus $F$ is equivalent to $B$.

\textbf{Necessity.} Assume that $F$ is equivalent to the Smith normal form $B$ of $F$.  It is obvious that $J_k(B)=\mathbf{R}$ for $k=1,2,\cdots,l$,
by Lemma 3.3, $J_k(F)=J_k(B)=\mathbf{R}$ for $k=1,2,\cdots,l$.
\end{proof}

\begin{theorem} Let $F\in \mathbf{R}^{l\times m}$ be of normal rank $r$, $d_r(F)=d_1^pd_2^q$, where $d_1=x_1-f_1(x_2,\cdots,x_n),d_2=x_2-f_2(x_3,\cdots,x_n)$, $p,q$ is a positive integer. Then $F$ is equivalent to its  Smith normal form iff $J_k(F)=\mathbf{R}$ for $k=1,2,\cdots,r$.
\end{theorem}
\begin{proof}
 \textbf{Sufficiency.} ~Since $d_r(F)=d_1^pd_2^q$, firstly, $d_i(F)\mid d_r(F)$, we can assume that $d_i(F)=d_1^{p_i}d_2^{q_i}$, where $p_i,q_i$ is a non-negative integer, $i=1,2,\cdots,r$. Let
 $$B=\begin{bmatrix}
            D & 0_{r\times(m-r)} \\
            0_{(l-r)\times r} & 0_{(l-r)\times(m-r)} \\
          \end{bmatrix},
 $$
where $D=diag\{D=diag\{d_1^{r_1}d_2^{s_1},d_1^{r_2}d_2^{s_2},\cdots,d_1^{r_l}d_2^{s_l}\}$, $s_1=p_1,r_1=q_1$, $s_i=p_i-p_{i-1},r_i=q_i-q_{i-1}$, $i=2,3,\cdots,r$.
From Lemma 3.14, there are $G\in \mathbf{R}^{l\times r}$, $F_1\in\mathbf{R}^{r\times m} $ such that $F=G\cdot F_1$, and $F_1$ is ZLP. By Quillen-Suslin Theorem, there is a unimodular matrix $Q_1\in\mathbf{R}^{m\times m}$ such that $F_1\cdot Q_1=
                       [ E_r~ ~ 0_{r\times(m-r)}]
$.
Then $F\cdot Q_1=(G\cdot F_1)\cdot Q_1=G\cdot(F_1\cdot Q_1)=G\cdot
                        [E_r~ ~0_{r\times(m-r)}]
                      =[G~ ~ 0_{r\times(m-r)}]
                      $, $d_k(G)=d_k(F\cdot Q_1)$, $J_k(G)=J_k(F\cdot Q_1)$. By Lemma 3.7, $d_k(G)=d_k(F)=d_1^{p_k}d_2^{q_k}$, $J_k(G)=J_k(F)=\mathbf{R}$, where $k=1,2,\cdots,r$. By Theorem 4.4, $G^{T}$ is equivalent to its  Smith normal form
                        $[D ~~0_{r\times(l-r)}]
                      $, then $G$ is equivalent to its Smith normal form
$\begin{bmatrix}
                          D \\
                          0_{(l-r)\times r} \\
                        \end{bmatrix},
                    $
i.e., there are unimodular matrices $P_1\in \mathbf{R}^{l\times l}$, $Q_2\in\mathbf{R}^{r\times r}$ such that
$$G=P_1\cdot\begin{bmatrix}
                          D \\
                          0_{(l-r)\times r} \\
                        \end{bmatrix}\cdot Q_2 .$$
Setting
$$P_2=\begin{bmatrix}
             Q_2 ~&  \\
              ~& E_{m-r} \\
           \end{bmatrix}.
$$
Then
\begin{align*}F\cdot Q_1&=G\cdot F_1\cdot Q_1=G\cdot (F_1\cdot Q_1)
                                   =G\cdot
                              [ E_r~  ~0_{r\times(m-r)}]\\
&=P_1\cdot\begin{bmatrix}
            D  \\
            0_{(l-r)\times r}  \\
          \end{bmatrix}\cdot Q_2\cdot
             [E_r ~~  0_{r\times (m-r)}]\\
&=P_1\cdot\begin{bmatrix}
            D Q_2 &0_{r\times (m-r)} \\
            0_{(l-r)\times r}&0_{(l-r)\times (m-r)}  \\
          \end{bmatrix}
\\&=P_1\cdot\begin{bmatrix}
            D~ &~ 0_{r\times(m-r)} \\
            0_{(l-r)\times r} ~&~ 0_{(l-r)\times(m-r)} \\
          \end{bmatrix}\cdot\begin{bmatrix}
             Q_2 &  \\
              & E_{m-r} \\
           \end{bmatrix}\\
&=P_1\cdot B\cdot P_2.
\end{align*}
Thus $F$ is equivalent to $B$.

\textbf{Necessity.} ~~Assume ~that ~$F$ ~is ~equivalent~ to ~its Smith normal form $B$, ~it is obvious that $J_k(B)=\mathbf{R}$ for $k=1,~2,~\cdots,~r$, by Lemma 3.2, $J_k(F)=J_k(B)=\mathbf{R}$ for $k=1,2,\cdots,r$.
\end{proof}

\begin{example}\quad Consider the following  $3D$ polynomial matrix\\
\\$$F(x_1,x_2,x_3)=\begin{bmatrix}
           a_{11}~ & a_{12} ~& a_{13} \\
          a_{21} ~& a_{22}~& a_{23} \\
           a_{31}~& a_{32} ~&a_{33}\\
           \end{bmatrix},$$
           
\end{example}
where
\begin{align*}
a_{11}&=1,~~~ a_{12}=x_1x_2,~~~a_{13}=x_1+x_2,\\
a_{21}&= -4x_2^2,~~~a_{22}=x_1x_2-x_1+x_2-x_2^2-4x_1x_2^3,\\
a_{23}&=x_2^2-x_1x_2-5x_2^3-3x_1x_2^2,~~~a_{31}=x_1x_2,\\
a_{32}&=x_1^2x_2-x_1^2-x_1x_2^2+x_1x_2+x_1^2x_2^2,\\
a_{33}&=3x_1x_2^2-2x_1x_2+x_1-x_2-x_2^3+2x_2^2+x_1^2x_2^2-x_1x_2^3.
\end{align*}
By Computing, $d_1(F)=1$, $d_2(F)=(x_1-x_2)(x_2-1)$, $d_3(F)=(x_1-x_2)(x_2-1)^2$. Then, computing the reduced Gr\"{o}bner bases of the ideal that generated by the reduced $k$-$th$ order minors of $F(\boldsymbol x)$, we have that $J_k(F)=\mathbf{R}$, $k=1,2,3$. From Theorem 2.3, $F(\boldsymbol x)$ is equivalent to its  Smith normal form $P(\boldsymbol x)$
$$P(\boldsymbol x)=\begin{bmatrix}
           1 ~&  &  \\
            & ~d_1d_2 &  \\
            &  & ~d_1d_2^2 \\
         \end{bmatrix},
$$
where $d_1=x_1-x_2,d_2=x_2-1$.\\
Considering
$$F(x_1,1)=\left(
                     \begin{array}{ccc}
                       1 & x_1 & x_1+1\\
                       -4 & -4x_1 & -4-4x_1 \\
                       x_1 & x_1^2 & x_1^2+x_1 \\
                     \end{array}
                   \right).
$$
We have $$U_1(x_1)\cdot F(x_1,1)=\left(
                                       \begin{array}{ccc}
                                       1 & x_1 & x_1+1 \\
                                         0 & 0 & 0 \\
                                         0& 0 & 0 \\
                                       \end{array}
                                     \right),
$$
where $$
U_1(x_1)=\left(
               \begin{array}{ccc}
                 1 & 0 & 0 \\
                 4 & 1 & 0 \\
                 -x_1 & 0 & 1\\
               \end{array}
             \right).
$$
Then $$\begin{aligned}U_1(x_1)\cdot F(\boldsymbol x)
=\begin{bmatrix}
                                           1&  &  \\
                                            & d_2 &  \\
                                            & & d_2 \\
                                          \end{bmatrix}\end{aligned} \cdot F_1$$
                                          
                               where $$F_1= \begin{aligned} \begin{bmatrix}
                                                        1 & x_1x_2 & x_1+x_2 \\
                                                        -4x_2-4 & x_1-x_2-4x_1x_2-4x_1x_2^2  & -5x_2^2-3x_1x_2-4x_2-4x_1 \\
                                                        x_1 & x_1^2x_2-x_1x_2+x_1^2 & x_1^2 x_2-x_2^2-x_1x_2^2+\\&&2x_1x_2+x_1^2+x_2-x_1 \\
                                                      \end{bmatrix},\end{aligned}
$$
Setting

                                                    $$  U_2=\begin{bmatrix}
                                                                      1 ~& 0~ & 0 \\
                                                                      4+4x_2 ~& 1 ~& 0 \\
                                                                      -x_2 ~& 0~& 1 \\
                                                                    \end{bmatrix}.
                                                    $$
Then
 $$U_2\cdot F_1=\begin{bmatrix}
                                             1&  &  \\
                                            & d_1 &  \\
                                           & & d_1 \\
                                          \end{bmatrix}\cdot\begin{bmatrix}
                                                         1 & x_1x_2 & x_1+x_2 \\
                                                        0 & 1 & x_2 \\
                                                        1 & x_1x_2+x_1 & x_1+2x_2+x_1x_2-1 \\
                                                      \end{bmatrix}.
                                        $$
Setting $$F_2(\boldsymbol x)=\begin{bmatrix}
                                                       1 & x_1x_2 & x_1+x_2 \\
                                                        0 & 1 & x_2 \\
                                                        1 & x_1x_2+x_1 & x_1+2x_2+x_1x_2-1 \\
                                                      \end{bmatrix} ,$$$$U_3=\begin{bmatrix}
                                                                      1 ~& 0~ & 0 \\
                                                                      0 ~& 1 ~& 0 \\
                                                                      -1 ~& -x_1~& 1 \\
                                                                    \end{bmatrix}.$$
it is obvious that $$U_3\cdot F_2(\boldsymbol x)=\begin{bmatrix}
                              1&  &  \\
                              & 1 & \\
                               &  & d_2 \\
                           \end{bmatrix}\cdot F_3(\boldsymbol x),
$$
where $$F_3(\boldsymbol x)=\begin{bmatrix}
                                                        1 & x_1x_2 & x_1+x_2 \\
                                                        0 & 1 & x_2 \\
                                                        0& 0 & 1 \\
                                                      \end{bmatrix}.$$
We can easily checked that $F_3$ is a unimodular matrix, and
\begin{align*}\begin{bmatrix}
                                1&  &\\
                                  & d_2&  \\
                                   &  & d_2 \\
                           \end{bmatrix}\cdot U_2^{-1}\cdot\begin{bmatrix}
                               1& & \\
                                  & d_1 & \\
                                  &  & d_1 \\
                           \end{bmatrix}\cdot U_3^{-1}\cdot\begin{bmatrix}
                               1& & \\
                                  & 1 & \\
                                  &  & d_2 \\
                           \end{bmatrix}
=V_1^{-1}\cdot\begin{bmatrix}
                             1& & \\
                              & d_1d_2 &\\
                              &  & d_1d_2^2\\
                          \end{bmatrix}.
\end{align*}
Then $$\begin{aligned}
F(\boldsymbol x)&=U_1^{-1}\cdot V_1^{-1}\cdot\begin{bmatrix}
                             1& & \\
                              & d_1d_2 &\\
                              &  & d_1d_2^2\\
                          \end{bmatrix} \cdot F_3(x_1,x_2).
                         \end{aligned}$$
Thus $$F(\boldsymbol x)=U(\boldsymbol x)\cdot\begin{bmatrix}
                              1& & \\
                              & d_1d_2 &\\
                              &  & d_1d_2^2\\
                           \end{bmatrix}\cdot V(\boldsymbol x),$$
where $U(\boldsymbol x)=U_1^{-1}\cdot V_1^{-1}$, and $V(\boldsymbol x)=F_3(\boldsymbol x)$ are unimodular matrices.

\section{Conclusions}
In this paper, we have investigated the quasi weakly linear multivariate matrices $F(\boldsymbol x)$, i.e., the g.c.d. of maximal order minors of $F(\boldsymbol x)$ is $(x_1-f_1(x_2,\cdots,x_n))^p(x_2-f_2(x_3,\cdots,x_n))^q$. Using hierarchical-recursive methods, we have obtained  $F(\boldsymbol x)$ is equivalent to its Smith normal form if and only if $J_i(F)=\mathbf{R}$ for $i=1,2,\cdots,k$, where $k$ is the normal rank of $F(\boldsymbol x)$. And $J_i(F)=\mathbf{R}$ if and only if the reduced Gr\"{o}bner bases of the ideal that generated by the reduced k-th order minors of $F(\boldsymbol x)$ is \{1\}. Therefore, we can verify whether is such a matrix equivalent to its Smith normal form by computing the reduced Gr\"{o}bner bases of the relevant  ideals. Finally, we gave the corresponding algorithm.

\section*{Acknowledgments}

 This research was supported by the National Natural Science Foundation of China(12371507), and  the Scientific Research Fund of Hunan Province Education Department (22A0334).

\bibliographystyle{elsarticle-harv}

\bibliography{template_ref}

\end{document}